\documentclass[11 pt]{article}
\usepackage{color,amsthm,amsmath, array,amssymb,amscd,amsfonts,latexsym, url,wasysym}
\usepackage{graphicx}
\usepackage[colorlinks,allcolors=blue]{hyperref}
\usepackage{xcolor}
\usepackage{lmodern}
\usepackage{tikz}
\usetikzlibrary{decorations.pathreplacing}

\usepackage{ytableau}
\usepackage{mathtools}
\usepackage{listings}
\usepackage{longtable}
\usepackage{cite}

\newtheorem{theo}{Theorem}[section]

\newtheorem{lemm}[theo]{Lemma}

\title{\bf Branching rule on winding subalgebras of affine Kac-Moody algebras}
\author{Duc-Khanh Nguyen}

\date{}

\renewcommand{\min}{\mathrm{min}}
\renewcommand{\max}{\mathrm{max}}
\newcommand{\h}{\mathrm{h}}
\newcommand{\Aut}{\mathrm{Aut}}

\newfont{\gothic}{eufb10}
\begin{document}
\maketitle

\setcounter{section}{-1}
 
\begin{abstract}
In this paper, by using the Lakshmibai-Seshadri paths, we give a branching rule for representations of affine Kac-Moody algebras to their winding subalgebras. As a corollary, we can describe branching multiplicities in the language of paths. An analog of Steinberg's formula for branching multiplicities is also given.
\end{abstract}


\setcounter{tocdepth}{1}

\section{Introduction}
One of the most important problems in representation theory is understanding the decomposition of an irreducible module into the sum of irreducible modules of a subalgebra: What is the support of the decomposition and formula for the branching multiplicities? In the representation theory of Kac-Moody algebras, the answers to those questions were answered in some special cases: Let $\mathfrak{g}$ be a symmetrisable Kac-Moody algebra. In \cite{Peter}, Peter Littelmann gave the generalized Littlewood-Richardson rule for the case $\mathfrak{g} \subset \mathfrak{g} \otimes \mathfrak{g}$ and the branching rule for the case $\mathfrak{h} \subset \mathfrak{g}$ when $\mathfrak{h}$ is a Levi subalgebra of $\mathfrak{g}$ in the language of Lakshmibai-Seshadri paths. Robert Steinberg gave the combinatorial formula for tensor multiplicities of simple Lie algebras in \cite{Steinberg}. It is similar to the analog Kostant's multiplicity formula for symmetrisable Kac-Moody algebras (see \cite{Carter}). 

In this paper, we study the branching rule for winding subalgebras $\mathfrak{g}[u]$ of affine Kac-Moody algebras $\mathfrak{g}$. The couple was studied in \cite{KacWakimoto} because of its relation with the tensor product case and a solution to Frenkel's conjecture \cite{frenkel2006representations}. In \cite{KacWakimoto}, Kac and Wakimoto gave the formula of the branching function and its modular transformation for both cases $\mathfrak{g}\subset \mathfrak{g} \otimes \mathfrak{g}$ and $\mathfrak{g}[u]\subset \mathfrak{g}$. The branching function for tensor products can be expressed as the branching function for winding subalgebras for some special cases. It gives us an idea to understand the branching rule of an affine Kac-Moody algebra to its winding subalgebras using the method used for the tensor product case. Previously, the author studied some particular cases in \cite{nguyen2020branching}. The result of this work is more general compared to \cite{nguyen2020branching} and more combinatorial compared to \cite{KacWakimoto}. It can be stated as follows. Let $P_+$ be the set of dominant integral weights. For each $\lambda\in P_+$, let $L(\lambda)$ be the irreducible $\mathfrak{g}$-module of highest weight $\lambda$. Let $\dot{P_+} $ and $\dot{L}(\lambda')$ where $\lambda' \in \dot{P_+}$ be the corresponding objects of $\mathfrak{g}[u]$. For any $\lambda\in P_+$ and $\lambda'\in \dot{P}_+$, let $c_{\lambda}^{\lambda'}$ be the multiplicity of $\dot{L}(\lambda')$ in the decomposition  \begin{equation*}L(\lambda)=\bigoplus\limits_{\lambda'\in\dot{P}_+}c_{\lambda}^{\lambda'}\dot{L}(\lambda').
\end{equation*}

\noindent For each dominant integral weight $\lambda$ of $\mathfrak{g}$, let $\wp_{\lambda}$ be the set of all Lakshmibai-Seshadri paths of shape $\lambda$ defined in \cite{Peter}. The paths are piecewise linear with starting point $\pi(0)=0$ and end point $\pi(1)$. A path is called $\mathfrak{g}[u]$-dominant if its image is contained in the dominant Weyl chamber of the root system of $\mathfrak{g}[u]$. The branching rule for $\mathfrak{g}[u]\subset \mathfrak{g}$ can be stated as follows. 

\begin{theo}Let $\mathfrak{g}$ be an affine Kac-Moody Lie algebra and $\mathfrak{g}[u]$ be a winding subalgebra of $\mathfrak{g}$. Then for any $\lambda\in P_+$, we have \begin{equation*}
L(\lambda)=\bigoplus\limits_{\pi}\dot{L}(\pi(1)), 
\end{equation*}
where the sum runs over all $\mathfrak{g}[u]$-dominant paths in $\wp_{\lambda}$. In particular, for each $\lambda' \in \dot{P}_+$ the multiplicity $c_\lambda^{\lambda'}$ is equal to the number of $\mathfrak{g}[u]$-dominant paths in $\wp_\lambda$ with the end point $\lambda'$. 
\end{theo}

\noindent Let $W$  be the Weyl group and $\rho$ be the sum of fundamental weights of $\mathfrak{g}$. Let $\dot{W},\dot{\rho}$ be the corresponding objects for $\mathfrak{g}[u]$. Let $\mathfrak{h}$ be a fixed Cartan subalgebra of $\mathfrak{g}$. For each root $\alpha \in \mathfrak{h}^*$, let $m_\alpha$ be its multiplicity. We define the generalized Kostant's partition function $\mathcal{P}:\mathfrak{h}^* \rightarrow \mathbb{N}$ as the number of ways of writing $\lambda$ as a sum of positive roots, each such root $\alpha$ being taken $m_\alpha$ times. We have a combinatorial expression for $c_\lambda^{\lambda'}$.

\begin{theo}Let $\mathfrak{g}$ be an affine Kac-Moody algebra and $\mathfrak{g}[u]$ be a winding subalgebra of $\mathfrak{g}$. Then for any $\lambda\in P_+$ and $\lambda' \in \dot{P}_+$, we have \begin{equation*} c_\lambda^{\lambda'}=\sum\limits_{\sigma\in W}\sum\limits_{\tau\in\dot{W}}\epsilon(\sigma\tau)\mathcal{P}(\sigma(\lambda+\rho)+\tau(\dot{\rho})-(\lambda'+\rho+\dot{\rho})).
\end{equation*} 
\end{theo}
The paper is organized as follows. In Section \ref{preliminaries}, we introduce the fundamental results of affine Kac-Moody algebras, winding subalgebras, and Lakshmibai-Seshadri paths. In Section \ref{result}, we prove our main theorems.
 
\section{Preliminaries}\label{preliminaries}
\subsection{Symmetrisable Kac-Moody algebras} In this subsection, we recall some fundamental notions and results for symmetrisable Kac-Moody algebras. Let $\mathfrak{g}$ be a symmetrisable Kac-Moody algebra. Let $W$ be the Weyl group. Let $P_+$ be the set of all dominant integral weights. Let $\rho$ be the sum of fundamental weights $\Lambda_i$. For each $\lambda \in P_+$, let $L(\lambda)$ be the irreducible highest weight $\mathfrak{g}$-module of highest weight $\lambda$. The character $ch_\lambda$ of $L(\lambda)$ is given by Kac.
\begin{theo}[Corollary 19.18, \cite{Carter}]\label{Kachar} Let $L(\lambda)$, $\lambda\in P_+$, be an irreducible module for a symmetrisable Kac-Moody algebra. Then \begin{equation*}
ch_\lambda=\frac{\sum\limits_{w\in W} \epsilon(w)e_{w(\lambda+\rho)}}{\sum\limits_{w\in W}\epsilon(w)e_{w(\rho)}}.
\end{equation*}
\end{theo}

\noindent Let $\mathfrak{h}$ be Cartan subalgebra of $\mathfrak{g}$. Let $\Phi$ be the root system of $\mathfrak{g}$ and $\Phi_+$ be the set of all positive roots. For each $\alpha\in\Phi$, we denote its multiplicity by $m_\alpha$. We define the generalized Kostant's partition function $\mathcal{P}:\mathfrak{h}^* \rightarrow \mathbb{N}$ by \begin{equation*}\mathcal{P}(\zeta)= \# \left\lbrace(r_{\alpha,i})_{\alpha\in \Phi_+, 1\leq i\leq m_\alpha} \mid r_{\alpha,i} \in \mathbb{N} \text{ and } \zeta =\sum\limits_{\alpha\in\Phi_+}\sum\limits_{i=1}^{m_\alpha} r_{\alpha,i}\alpha \right\rbrace
\end{equation*} for each $\zeta \in \mathfrak{h}^*$. The multiplicity $m_\lambda(\mu)$ of each weight $\mu$ of $L(\lambda)$ is then given by a combinatorial formula.
\begin{theo}[Proposition 19.20, \cite{Carter}]\label{multpart} Let $\mathfrak{g}$ be a symmetrisable Kac-Moody algebra and $\lambda\in P_+$. Then for each weight $\mu$ of $L(\lambda)$ we have \begin{equation*}
m_\lambda(\mu) = \sum\limits_{w\in W}\epsilon(w)\mathcal{P}(w(\lambda+\rho)-(\mu+\rho)).
\end{equation*}
\end{theo}
\subsection{Affine Kac-Moody algebras and its winding subalgebras} In this subsection, we consider the case $\mathfrak{g}$ is an affine Kac-Moody algebra. Let $I:=\{0,\dots,l\}$ and $A:=(a_{ij})_{i,j\in I}$ to be the generalized Cartan matrix defining $\mathfrak{g}$. Let $a:=\, \,^t(a_0,\dots,a_l)$ and $c:=(c_0,\dots, c_l)$ be the null vectors of relative prime integers such that $a_i,c_i\geq 0 $ and $Aa=cA=0$. The Coxeter number and dual Coxeter number of $\mathfrak{g}$ are $\h:=\sum\limits_{i\in I}a_i$ and $\h^{\vee}:=\sum\limits_{i\in I}c_i$, respectively. Let $(\mathfrak{h},\Pi,\Pi^{\vee})$ with $\Pi:=\{\alpha_0,\dots,\alpha_l\}, \Pi^\vee :=\{h_0,\dots,h_l\}$ be a realisation of $\mathfrak{g}$. Let $K :=\sum\limits_{i\in I} c_ih_i$ be the canonical central element and $\delta :=\sum\limits_{i\in I} a_i\alpha_i$ be the basis imaginary root of $\mathfrak{g}$. \\

The winding subalgebras of $\mathfrak{g}$ is defined in \cite{KacWakimoto, Wakimoto} as follows: Fix $u\in \mathbb{Z}_{>0}$ relative prime to $a_0$. Let $\Pi_u:=\{\dot{\alpha_0},\dots,\dot{\alpha_l}\}$ and $\Pi_u^\vee:=\{\dot{h_0},\dots,\dot{h_l}\}$, where \begin{equation*}\dot{\alpha_0}:=\frac{u-1}{a_0}\delta+\alpha_0, \dot{\alpha_i}:=\alpha_i \text{ for } i>0,\end{equation*} \begin{equation*} \dot{h_0}:=\frac{u-1}{c_0}K+h_0, \dot{h_i}:=h_i \text{ for } i>0.\end{equation*} The winding subalgebra $\mathfrak{g}[u]$ of $\mathfrak{g}$ which is defined by realization $(\mathfrak{h},\Pi_u,\Pi_u^{\vee})$. Since the matrix $(\dot{\alpha_i}(\dot{h_j}))_{i,j\in I}$ is exactly $A$, the subalgebra $\mathfrak{g}[u]$ is isomorphic to the algebra $\mathfrak{g}$. By indicating objects associated to $\mathfrak{g}[u]$ with an overdot, we have: 
\begin{itemize}
\item[] $\dot{K}=\sum_Ic_i\dot{h_i}=uK, \dot{\delta}=\sum_Ia_i\dot{\alpha_i}=u\delta,$ 
\item[] $\dot{\h}=\h,\dot{\h}^\vee =\h^\vee,$
\item[] $\dot{\Lambda_i}=\Lambda_i+(\frac{1}{u}-1)\frac{c_i}{c_0}\Lambda_0,$ 
\item[] $\dot{\rho}=\sum_I \dot{\Lambda_i}=\rho + (\frac{1}{u}-1)\h^\vee\Lambda_0$,
\item[] $\dot{W}=\langle \dot{s_0}, \dots,\dot{s_l} \rangle$, where $\dot{s_i}\in \Aut(\mathfrak{h}^*), \dot{s_i}(\lambda)=\lambda-\lambda(\dot{h_i})\dot{\alpha_i},$
\item[] $\dot{ch}_\lambda=\frac{\sum\limits_{w\in \dot{W}} \epsilon(w)e_{w(\lambda+\dot{\rho})}}{\sum\limits_{w\in \dot{W}}\epsilon(w)e_{w(\dot{\rho})}}$ for any $\lambda \in\dot{P}_+$.
\end{itemize}

\subsection{Lakshmibai-Seshadri paths and character formula}   
In this section, we recall the fundamental theory of Lakshmibai-Seshadri paths and the character formula in the language of paths for any symmetrisable Kac-Moody algebra in \cite{Peter}. Let $(.|.)$ be the standard bilinear symmetric invariant form of $\mathfrak{g}$. For each real root $\alpha$, we denote $\alpha^\vee :=\frac{2 \alpha}{(\alpha |\alpha)}$ its real coroot. Let $l(.)$ be the length function on $W$. Let $s_\beta$ be the reflection corresponding to a positive real root $\beta$.  
\subsubsection{Path operators}

Let $P_{\mathbb{R}} :=P\otimes_{\mathbb{Z}}\mathbb{R}$. Denote $\wp$ the set of all piecewise linear paths $\pi: [0,1]\rightarrow P_{\mathbb{R}} $ such that $\pi(0)=0$ with the equivalent relation $\pi \sim \pi'$ if $\pi=\pi'$ up to a reparametrization. For $\pi_1,\pi_2 \in \wp$, we define $\pi_1 * \pi_2 $ to be a path still in $\wp$ by \begin{equation*}
(\pi_1 *\pi_2)(t):= \begin{cases} \pi_1(2t) &\text{ if } t\in [0,\frac{1}{2}]; \\
\pi_1(1)+\pi_2(2t-1) &\text{ if }t\in[\frac{1}{2},1]. \end{cases}
\end{equation*} For each simple root $\alpha$ and a path $\pi$ in $\wp$, we define the path $s_\alpha(\pi)$ still in $\wp$ by \begin{equation*}s_\alpha(\pi)(t):=s_\alpha(\pi(t)).
\end{equation*}  
We define operators $e_\alpha$ and $f_\alpha$ on $\wp \cup \{0\}$ as follows.
\begin{itemize}

\item[1.] Let $h_\alpha: [0,1]\rightarrow \mathbb{R}$ be a function defined by $h_\alpha(t):= (\pi(t)|\alpha^\vee)$. Set \begin{equation*}Q:=\min_{t\in[0,1]}\{h_\alpha(t)\cap \mathbb{Z}\} \end{equation*} and \begin{equation*}q:=\min\{t \in [0,1]\text{ such that }h_\alpha(t)=Q
\},\end{equation*} \begin{equation*}p:=\max \{t\in [0,1] \text{ such that }h_\alpha(t)=Q\}.
\end{equation*}
Note that $Q \leq 0$ since $\pi(0)=0$ and $0\leq [ h_\alpha(1)-Q]$ where $[.]$ means the integral part of a real number \begin{equation*}[r]=\begin{cases} \lfloor x \rfloor &\text{ if }x \geq 0; \\ 
\lceil x \rceil &\text{ if } x \leq 0.
 \end{cases}
 \end{equation*}

\item[2.]If $Q<0$, let $y$ be the unique value of $t$ in $[0,q)$ such that $h_\alpha(y)=Q+1$ and $Q<h_\alpha(t)<Q+1$ for all $t\in (y,q)$. In this case, let $\pi_1, \pi_2,\pi_3$ be three paths in $\wp$ defined by 
\begin{equation*}\pi_1(t):=\pi(ty); \pi_2(t):=\pi(y+t(q-y))-\pi(y);\pi_3(t):=\pi(q+t(1-q))-\pi(q).
\end{equation*} 
In the other words, $\pi_1,\pi_2,\pi_3$ are obtained from restriction of the path $\pi$ on intervals $[0,y],[y,q],[q,1]$, then take translations of points $\pi(0),\pi(y),\pi(q)$ to $0$. We define the operator $e_\alpha$ on $\wp \cup \{0\}$ by \begin{equation*}e_\alpha(\pi):=\begin{cases} 0 &\text{ if }Q=0;\\
\pi_1 * s_\alpha(\pi_2) * \pi_3 &\text{ if } Q< 0.
\end{cases}
\end{equation*}

\item[3.]If $[ h_\alpha(1)-Q]>0 $, let $x$ be the unique value of $t$ in $(p,1]$ such that $h_\alpha(x)=Q+1$ and $Q<h_\alpha(t)<Q+1$ for all $t\in (p,x)$. In this case, let $\hat{\pi}_1, \hat{\pi}_2,\hat{\pi}_3$ be three paths in $\wp$ defined by \begin{equation*}\hat{\pi}_1(t):=\pi(tp); \hat{\pi}_2(t):= \pi(p+t(x-p))-\pi(p);\hat{\pi}_3(t):=\pi(x+t(1-x))-\pi(x).
\end{equation*} In the other words, $\hat{\pi}_1,\hat{\pi}_2,\hat{\pi}_3$ are obtained from restriction of the path $\pi$ on intervals $[0,p],[p,x],[x,1]$, then take translations of points $\pi(0),\pi(p),\pi(x)$ to $0$. We define the operator $f_\alpha$ on $\wp \cup \{0\}$ by \begin{equation*}
f_\alpha(\pi):=\begin{cases} 0 &\text{ if }[ h_\alpha(1)-Q ]=0 ;\\
\hat{\pi}_1 * s_\alpha(\hat{\pi}_2) * \hat{\pi}_3 &\text{ if } [ h_\alpha(1)-Q ]>0 .
\end{cases}
\end{equation*}
\end{itemize}
\begin{figure}[htbp]
  \centering
\includegraphics[width=10cm]{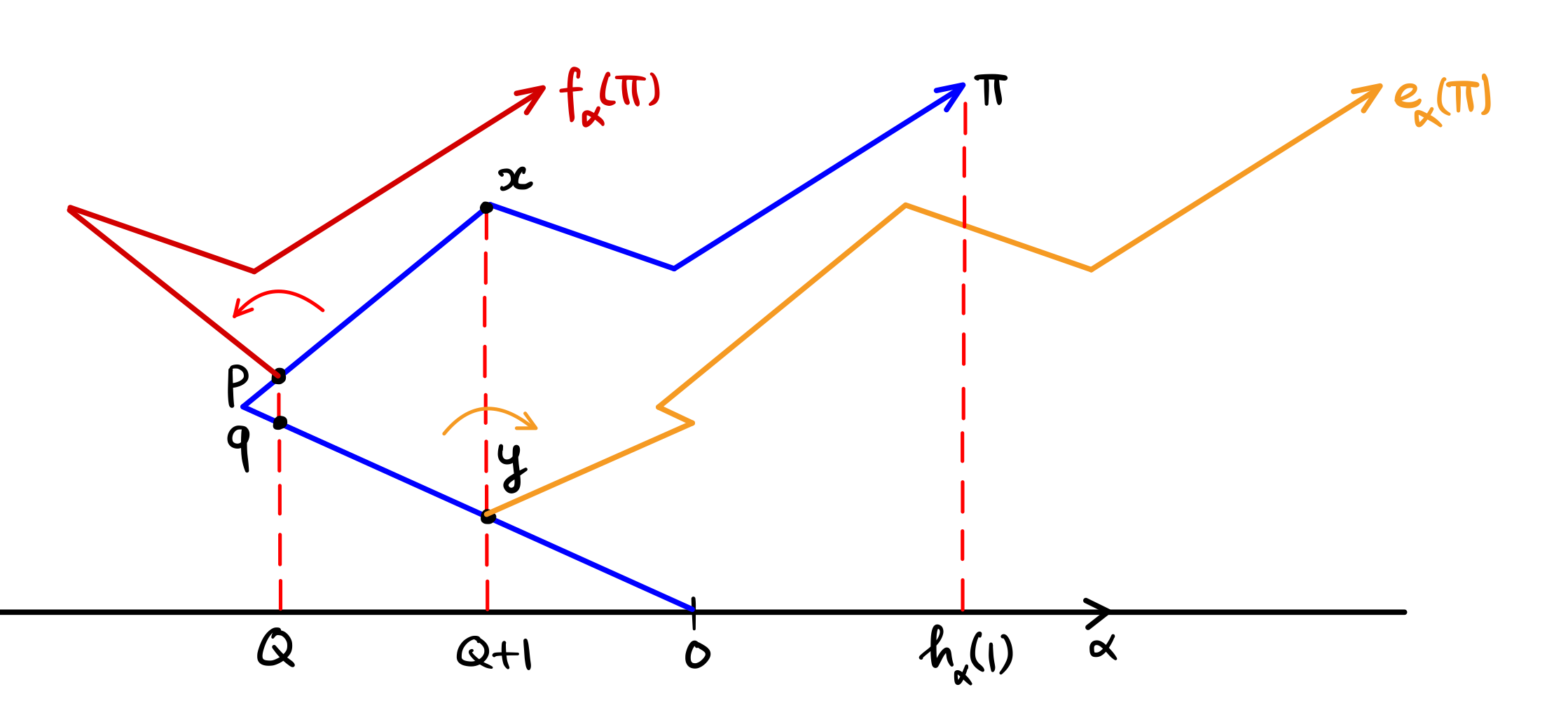}
\end{figure}
\begin{lemm}[Lemma 1.4, \cite{Peter}]\label{duoi}
\hfill\\
If $e_\alpha(\pi) \ne 0 $ then $e_\alpha(\pi)(1)=\pi(1)+\alpha$. If $f_\alpha(\pi) \ne 0$ then $f_\alpha(\pi)(1)=\pi(1)-\alpha$. 
\end{lemm}

\begin{lemm}[Proposition 1.5, \cite{Peter}]\label{maxminef}\hfill\\
$e_\alpha^n(\pi)=0$ if and only if $n>-Q$ and $f_\alpha^n(\pi)=0$ if and only if $n>[h_\alpha(1)-Q]$.
\end{lemm}

\subsubsection{Lakshmibai-Seshadri paths and character formula}
For each $\lambda\in P_+$, let $W_\lambda$ the stabilizer of $\lambda$ and $\leq$ the Bruhat order on $W/W_\lambda$. We denote $(\underline{\tau},\underline{a})$ the pair of a sequence  in $W/W_\lambda$ $$\underline{\tau}:\tau_1 > \tau_2 > \dots > \tau_r$$  and a sequence in $\mathbb{Q}$ $$\underline{a}: 0=a_0<a_1<\dots <a_r=1.$$ We associate this pair with a piecewise linear path $\pi^\lambda_{(\underline{\tau},\underline{a})}:[0,1]\rightarrow P_{\mathbb{R}}$ defined by \begin{equation*}
\pi^\lambda _{(\underline{\tau},\underline{a})}(t) :=\sum\limits_{i=1}^{j-1}(a_i-a_{i-1})\tau_i(\lambda)+(t-a_{j-1})\tau_j(\lambda)\text{ for } a_{j-1}\leq t \leq a_j.
\end{equation*} We call such paths like this a rational $W$-path of shape $\lambda$. For each pair $(\tau, \sigma)$ in $W/W_\lambda$ with $\tau >\sigma$ and a rational number $a$ in $(0,1)$. We say an $a$-chain of $(\tau,\sigma)$ a sequence in $W/W_\lambda$ $$\kappa_0:= \tau >\kappa_1:= s_{\beta_1}\tau > \kappa_2 := s_{\beta_2}s_{\beta_1}\tau >\dots >\kappa_s:= s_{\beta_s}...s_{\beta_1}\tau  =\sigma, $$ where $\beta_1,\dots,\beta_s$ are positive real roots such that $a (\kappa_i(\lambda)|\beta_i^\vee) \in \mathbb{Z}$ and $l(\kappa_i)-1 =l(\kappa_{i-1})$.
\noindent For each $\lambda\in P_+$, A Lakshmibai-Seshadri path of shape $\lambda $ is defined to be an $W$-path of shape $\lambda$ $\pi^\lambda_{(\underline{\tau},\underline{a})}$ such that there exists an $a_i$-chain for the pair $(\tau_i,\tau_{i+1})$ for all $i =1,...,r-1$. The set of all Lakshmibai-Seshadri paths of shape $\lambda$ is denoted by $\wp_\lambda$.  
\begin{lemm}[Lemma 3.5 and Proposition 4.2, \cite{Peter}]\label{Qmin} We have $\wp_\lambda \subset \wp$ and the set of $\{0\}\cup \wp_\lambda$ is stable under the action of operators $e_\alpha$ and $f_\alpha$. Moreover, for each $\pi \in \wp_\lambda$, we have $\min_{t\in [0,1]}\{h_\alpha(t)\}$ is $Q$.
\end{lemm} 

\begin{theo}[Character formula, \cite{Peter}] \label{charpath} Let $\mathfrak{g}$ be a symmetrisable Kac-Moody algebra and $\lambda \in P_+$. Then the set $\wp_\lambda$ is equal to the set of all paths $\pi$ of the form \begin{equation*}
\pi= f_{\alpha_1} \circ f_{\alpha_2} \circ \dots \circ f_{\alpha_s}(\pi_\lambda), 
\end{equation*} where $\alpha_1,\dots,\alpha_s$ are simple roots and $\pi_\lambda$ is the straight line in $P_\mathbb{R}$ connecting $0$ with $\lambda$. Furthermore, \begin{equation*}ch_\lambda = \sum\limits_{\pi \in \wp_\lambda}e_{\pi(1)}.
\end{equation*}
\end{theo}

\section{The main results}\label{result} 
Let $\mathfrak{g}$ be an affine Kac-Moody algebra and $A$ the generalized Cartan matrix which defines $\mathfrak{g}$. Fix $u\in \mathbb{Z}_{>0}$ relative prime to $a_0$ in the null vector $a$ of $A$. Let $\mathfrak{g}[u]$ be the corresponding winding subalgebra. For each $k\in\mathbb{N}$, let $P_+^k$ be the set of all dominant integral weights of level $k$. For each integrable irreducible highest weight $\mathfrak{g}$-module $L(\lambda)$ of level $k>0$, we can decompose it into a direct sum of integrable irreducible highest weight modules $\dot{L}(\lambda')$ of level $uk$ (since $\dot{K}=uK$). Each $\dot{L}(\lambda')$ appears in $L(\lambda)$ with finite multiplicity, which we denote by $c_\lambda^{\lambda'}$ \begin{equation*}L(\lambda)=\bigoplus\limits_{\lambda'\in\dot{P}_+^{uk}}c_{\lambda}^{\lambda'}\dot{L}(\lambda').
\end{equation*}

\noindent A path $\pi \in \wp_\lambda$ is called $\mathfrak{g}[u]$-dominant if its image is contained in the dominant Weyl chamber of the root system of $\mathfrak{g}[u]$. We obtain a branching rule for $\mathfrak{g}[u]\subset \mathfrak{g}$ in the following theorem.
\begin{theo}\label{windingdecomp} Let $\mathfrak{g}$ be an affine Kac-Moody Lie algebra and $\mathfrak{g}[u]$ be a winding subalgebra of $\mathfrak{g}$. Then for any $\lambda\in P_+$, we have 
\begin{equation*}
L(\lambda)=\bigoplus\limits_{\pi}\dot{L}(\pi(1)),
\end{equation*}
where the sum runs over all $\mathfrak{g}[u]$-dominant paths in $\wp_{\lambda}$. In particular, for each $\lambda' \in \dot{P}_+$ the multiplicity $c_\lambda^{\lambda'}$ is equal the number of $\mathfrak{g}[u]$-dominant paths in $\wp_\lambda$ with end point $\lambda'$. 
\end{theo}
\begin{proof}
By using Kac's character formula for $\dot{ch}_{\lambda'}$, the equality \begin{equation*}
ch_\lambda =\sum\limits_{\lambda' \in \dot{P}_+^{uk}} c_{\lambda}^{\lambda'} \dot{ch}_{\lambda'}
\end{equation*} becomes 
\begin{equation}
\sum\limits_{\mu}m_\lambda(\mu)\sum\limits_{\sigma\in\dot{W}}\epsilon(\sigma)e_{\mu+\sigma(\dot{\rho})} = \sum\limits_{\lambda'\in \dot{P}_+^{uk}}c_\lambda^{\lambda'}\sum\limits_{\sigma\in\dot{W}}\epsilon(\sigma)e_{\sigma(\lambda'+\dot{\rho})}. \label{eq:char1'}
\end{equation}
Replace $\mu$ by $\sigma(\mu)$, then the left-hand side of the equation (\ref{eq:char1'}) does not change. So we have
\begin{equation}
\sum\limits_{\mu}m_\lambda(\mu)\sum\limits_{\sigma\in\dot{W}}\epsilon(\sigma)e_{\sigma(\mu+\dot{\rho})} = \sum\limits_{\lambda'\in \dot{P}_+^{uk}}c_\lambda^{\lambda'}\sum\limits_{\sigma\in\dot{W}}\epsilon(\sigma)e_{\sigma(\lambda'+\dot{\rho})}. \label{eq:char2'}
\end{equation}  We change the variable both sides of the equation (\ref{eq:char2'}) by putting $\eta = \sigma(\mu+\dot{\rho}) -\dot{\rho}$ on the left, $\eta = \sigma(\lambda'+\dot{\rho})-\dot{\rho}$ on the right, and then replacing $\sigma^{-1}$ by $\sigma$ to obtain 
\begin{equation}
\sum\limits_{\eta}\sum\limits_{\sigma\in\dot{W}} \epsilon(\sigma)m_\lambda(\sigma(\eta+\dot{\rho})-\dot{\rho})e_{\eta+\dot{\rho}} = \sum\limits_{\eta}\sum\limits_{\sigma\in\dot{W}}\epsilon(\sigma)c_\lambda^{\sigma(\eta+\dot{\rho})-\dot{\rho}}e_{\eta+\dot{\rho}}. \label{eq:char3'}
\end{equation}
As a weight of $\mathfrak{g}[u]$, if $\eta$ is dominant then $\eta +\dot{\rho}$ is strictly dominant. Then $\sigma(\eta+\dot{\rho})-\dot{\rho}$ is not dominant unless $\sigma=1$. Hence $c_{\lambda}^{\sigma(\eta +\dot{\rho})-\dot{\rho}}=0$ unless $\sigma =1$. So, if we restrict both sides of the equation (\ref{eq:char3'}) over $\eta \in \dot{P}_+$ and then delete $e_{\dot{\rho}}$ both sides we get 
\begin{equation}
\sum\limits_{\eta \in \dot{P}_+}\sum\limits_{\sigma\in\dot{W}} \epsilon(\sigma)m_\lambda(\sigma(\eta+\dot{\rho})-\dot{\rho})e_{\eta} = \sum\limits_{\eta \in \dot{P}_+}c_\lambda^\eta e_{\eta}. \label{eq:char4'}
\end{equation}
We see that if there exist $\sigma \in \dot{W}$ such that $\sigma(\mu+\dot{\rho})-\dot{\rho}$ is a dominant weight then $\sigma$ and $\eta$ are uniquely determined by $\mu$. Then we set $p(\mu):=\epsilon(\sigma)$ and $\{\mu\}:=\eta$. In the case there does not exist $\sigma \in \dot{W}$ satisfing above condition, we set $p(\mu):=0$ and $\{\mu\}:=0$. By Theorem \ref{charpath}, each weight $\mu$ of $L(\lambda)$ is an endpoint $\pi(1)$ of some path $\pi\in \wp_\lambda$. Furthermore, the multiplicity $m_\lambda(\mu)$ is equal to the number of Lakshmibai-Seshadri paths of shape $\lambda$ with endpoint $\mu$. So, in the language of paths, we can rewrite the equation (\ref{eq:char4'}) as 
\begin{equation*}
\sum\limits_{\pi\in \wp_\lambda}p(\pi(1))e_{\{\pi(1)\}}=\sum\limits_{\eta\in \dot{P}_+}c^\eta_\lambda e_\eta. \label{eq:char5'}
\end{equation*} 
By character formula in Theorem \ref{charpath}, we obtain 
\begin{equation}
\sum\limits_{\pi\in \wp_\lambda}p(\pi(1))\dot{ch}_{\{\pi(1)\}}=\sum\limits_{\eta\in\dot{P}_+}c_\lambda^\eta \dot{ch}_\eta. 
\label{eq:char6'}
\end{equation}
Let $\Omega$ be the set of all $\pi$ in $\wp_\lambda$ such that $\pi$ is $\mathfrak{g}[u]$-dominant and let $\Omega'$ be the completion of $\Omega$ in $\wp_\lambda$. We see that for each $\pi \in \Omega$, $\pi(1)$ is dominant then $p(\pi(1))=1$ and $\{\pi(1)\}=\pi(1)$. By Lemma \ref{cancel} below, we show that for each $\pi \in \Omega'$, there exist $\bar{\pi} \in \Omega'$ such that $p(\bar{\pi}(1))=-p(\pi(1))$ and $\{\bar{\pi}(1)\}=\{\pi(1)\}$. Hence by equation (\ref{eq:char6'}) we have 
\begin{equation*}
\sum\limits_{\pi\in \Omega}\dot{ch}_{\pi(1)}=\sum\limits_{\eta\in \dot{P}_+}c^\eta_\lambda \dot{ch}_\eta.
\end{equation*}
It implies that the set of all $\eta \in \dot{P}_+^{uk}$ such that $\dot{L}(\eta) \subset L(\lambda)$ as $\mathfrak{g}[u]$-modules is the set of endpoints of paths in $\Omega$. The multiplicity $c_\lambda^\eta$ is then equal to the number of paths in $\Omega$ with endpoint $\eta$. 
\end{proof}

\begin{lemm}\label{cancel} With the notation in the proof of Theorem \ref{windingdecomp}. For each $\pi \in \Omega'$, there exist $\bar{\pi}\in \Omega'$ such that $p(\bar{\pi}(1))=-p(\pi(1))$ and $\{\bar{\pi}(1)\}=\{\pi(1)\}$. 
\end{lemm} 
\begin{proof}
For each $\pi \in \Omega'$, there exist $t\in [0,1]$ such that $\pi(t)$ is not contained in the dominant Weyl chamber of $\mathfrak{g}[u]$. It means that there exist a simple root $\alpha$ of $\mathfrak{g}[u]$ and $t\in [0,1]$ such that $(\pi(t)|\alpha^\vee)<0$. By Lemma \ref{Qmin}, $\min_{t\in[0,1]}\{(\pi(t)|\alpha^\vee)\}=Q$ is an integer, we can choose $s$ to be the minimum value of $t\in [0,1]$ and a simple root $\alpha$ of $\mathfrak{g}[u]$ such that $(\pi(s)|\alpha^\vee)=-1$. Let $M_\pi$ be the subset of $\wp_\lambda$ of all paths $\pi'$ such that $\pi'(zt)= \pi(st)$ for some $z\in [0,1]$. Clearly that $M_\pi \subset \Omega'$ for each $\pi \in \Omega'$. By the minimum value of $s$, two sets $M_\pi, M_{\pi'}$ are the same set or have no common path for each $\pi,\pi' \in \Omega'$. Hence $\Omega'$ is the disjoint union of some subsets $M_\pi$ for $\pi \in \Omega'$.

Fix a path $\pi'\in\Omega'$. We need to show that for each $\pi \in M_{\pi'}$, we can find a path $\bar{\pi}$ in $M_{\pi'}$ that satisfies the statement of the lemma. Or equivalently, for each $(\pi,\sigma)\in M_{\pi'} \times \dot{W}$ such that $\sigma(\pi(1)+\dot{\rho})-\dot{\rho} \in \dot{P}_+$, there exist $(\bar{\pi},\bar{\sigma})\in M_{\pi'} \times \dot{W}$ such that $\epsilon(\sigma)=-\epsilon(\bar{\sigma})$ and $\sigma(\pi(1)+\dot{\rho})=\bar{\sigma}(\bar{\pi}(1)+\dot{\rho})$. This can be done by choosing $\bar{\sigma}=\sigma s_\alpha$ and $\bar{\pi} = f_\alpha^{1+(\pi(1)|\alpha^\vee)}(\pi)$. Indeed, we have $f_\alpha(M_\pi) \subset M_\pi \cup \{0\}$. Since $(\pi(1)|\alpha^\vee)+1 \leq (\pi(1)|\alpha^\vee)-Q$, we have $\bar{\pi}\ne 0$ by Lemma \ref{maxminef}. Finally, $\sigma(\pi(1)+\dot{\rho})=\bar{\sigma}(\bar{\pi}(1)+\dot{\rho})$ by Lemma \ref{duoi}.
\end{proof}

\begin{theo}Let $\mathfrak{g}$ be an affine Kac-Moody algebra and $\mathfrak{g}[u]$ be a winding subalgebra of $\mathfrak{g}$. Then for any $\lambda\in P_+^k$ and $\lambda' \in \dot{P}_+^{uk}$, we have \begin{equation*}
c_\lambda^{\lambda'}=\sum\limits_{\sigma\in W}\sum\limits_{\tau\in\dot{W}}\epsilon(\sigma\tau)\mathcal{P}(\sigma(\lambda+\rho)+\tau(\dot{\rho})-(\lambda'+\rho+\dot{\rho})).
\end{equation*}
\end{theo}
\begin{proof}
By Kac's character formula, the equality \begin{equation*}
ch_\lambda =\sum\limits_{\lambda' \in \dot{P}_+^{uk}} c_{\lambda}^{\lambda'} \dot{ch}_{\lambda'}
\end{equation*} becomes 
\begin{equation}ch_\lambda \sum\limits_{\sigma \in \dot{W}} \epsilon(\sigma)e_{\sigma(\dot{\rho})} = \sum\limits_{\lambda' \in \dot{P}_+^{uk}}c_{\lambda}^{\lambda'} \sum\limits_{\sigma\in \dot{W}}\epsilon(\sigma) e_{\sigma(\lambda'+\dot{\rho})}. \label{eq:char5} 
\end{equation} Now we use Theorem \ref{multpart}, to rewrite $ch_\lambda$ in the left-hand side of equation (\ref{eq:char5}), we get \begin{equation}\sum\limits_{\mu}\sum\limits_{\tau\in W}\sum\limits_{\sigma\in \dot{W}} \epsilon(\tau\sigma) \mathcal{P}(\tau(\lambda+\rho)-(\mu+\rho))e_{\sigma(\dot{\rho})+\mu}=\sum\limits_{\lambda' \in \dot{P}_+^{uk}}c_{\lambda}^{\lambda'} \sum\limits_{\sigma\in \dot{W}}\epsilon(\sigma) e_{\sigma(\lambda'+\dot{\rho})}.  \label{eq:char6}
\end{equation} We change the variable both sides of the equation (\ref{eq:char6}) by putting $\eta = \sigma(\dot{\rho})+\mu -\dot{\rho}$ on the left, $\eta = \sigma(\lambda'+\dot{\rho})-\dot{\rho}$ on the right and then replacing $\sigma^{-1}$ on the right-hand side by $\sigma$ to obtain \begin{equation}
\sum\limits_{\eta}\sum\limits_{\tau\in W}\sum\limits_{\sigma\in \dot{W}} \epsilon(\tau\sigma) \mathcal{P}(\tau(\lambda+\rho)+\sigma(\dot{\rho})-(\eta+\rho+\dot{\rho}))e_{\eta +\dot{\rho}}
=
\sum\limits_\eta\sum\limits_{\sigma \in \dot{W}}\epsilon(\sigma)c_\lambda^{\sigma(\eta+\dot{\rho})-\dot{\rho}} e_{\eta+\dot{\rho}}. \label{eq:char7}
\end{equation}
As a weight of $\mathfrak{g}[u]$, if $\eta$ is dominant then $\eta +\dot{\rho}$ is strictly dominant. Then $\sigma(\eta+\dot{\rho})-\dot{\rho}$ is not dominant unless $\sigma=1$. Hence, $c_{\lambda}^{\sigma(\eta +\dot{\rho})-\dot{\rho}}=0$ unless $\sigma =1$. So, if we restrict both sides of the equation (\ref{eq:char7}) over $\eta \in \dot{P}_+$ then we get 
\begin{equation*}\sum\limits_{\eta \in \dot{P}_+}\sum\limits_{\tau\in W}\sum\limits_{\sigma\in \dot{W}} \epsilon(\tau\sigma) \mathcal{P}(\tau(\lambda+\rho)+\sigma(\dot{\rho})-(\eta+\rho+\dot{\rho}))e_{\eta +\dot{\rho}}
=
\sum\limits_{\eta \in \dot{P}_+} c_\lambda^\eta e_{\eta+\dot{\rho}}.
\label{eq:char4}
\end{equation*} Note that $\eta$ in both sides belong to $\dot{P}_+^{uk}$. It implies the formula of $c_{\lambda}^\eta$ in terms of Weyl groups and generalized Konstant's partition function for any $\lambda$ in $P_+^k$ and $\eta \in \dot{P}_+^{uk}$.
\end{proof}

\addcontentsline{toc}{section}{References}
\bibliography{references}{}
\bibliographystyle{alpha}
\noindent Okinawa Institute of Science and Technology, Onna-son, Okinawa, Japan 904-0495 \\
E-mail: \href{mailto:khanh.mathematic@gmail.com}{khanh.mathematic@gmail.com} \\
\end{document}